\documentclass[reqno]{amsart}
\usepackage{amsmath}%
\usepackage{amsfonts}%
\usepackage{amssymb}%
\usepackage{graphicx}
\newtheorem{theorem}{Theorem}
\theoremstyle{plain}

\newtheorem{corollary}{Corollary}

\newtheorem{definition}{Definition}
\newtheorem{example}{Example}

\newtheorem{lemma}{Lemma}

\newtheorem{proposition}{Proposition}

\numberwithin{equation}{section}
\begin{document}
\begin{center}
\vspace*{1.3cm}

\textbf{MINIMIZERS OF GERSTEWITZ FUNCTIONALS}

\bigskip

by

\bigskip

PETRA WEIDNER\footnote{HAWK Hochschule f\"ur angewandte Wissenschaft und Kunst Hildesheim/\-Holzminden/ G\"ottingen University of Applied Sciences and Arts, Faculty of Natural Sciences and Technology,
D-37085 G\"ottingen, Germany, {petra.weidner@hawk-hhg.de}.}

\bigskip
\bigskip
Research Report \\ 
April 25, 2017

\end{center}

\bigskip
\bigskip

\noindent{\small {\textbf{Abstract:}}
Scalarization in vector optimization is essentially based on the minimization of Gerstewitz functionals. In this paper, the minimizer sets of Gerstewitz functionals are investigated. Conditions are given under which such a set is nonempty and compact. Interdependencies between solutions of problems with different parameters or with different feasible point sets are shown. Consequences for the parameter control in scalarization methods are derived. It is pointed out that the minimization of Gerstewitz functionals is equivalent to an optimization problem which generalizes the scalarization by Pascoletti and Serafini.
}

\bigskip


\bigskip


\section{Introduction}

Gerstewitz functionals were introduced by Gerstewitz (later Gerth, now Tammer) in the context of vector optimization \cite{ger85}. They have been investigated in \cite{GerWei90} and \cite{Wei90}, later followed by
\cite{GopRiaTamZal:03}, \cite{TamZal10} and \cite{DT}. Necessary and sufficient conditions for its basic properties under more general assumptions are given in \cite{Wei17a} and \cite{Wei17b}. In \cite{Wei17b}, it is shown that Gerstewitz functionals can represent orders, preference relations and other binary relations and thus act as a general tool for scalarization. This tool is used in multicriteria optimization, decision theory, mathematical finance, production theory and operator theory. In many cases, it is not obvious or not known that the function which is minimized is a Gerstewitz functional. A complete characterization of solutions in vector optimization by minimizers of Gerstewitz functionals has been given in \cite{Wei16b} and \cite{Wei16c}. 

The original version of the results presented in this paper was developed by the author in \cite{Wei90}. It has been extended using the properties of Gerstewitz functionals proved in \cite{Wei17a} and \cite{Wei17b}. 

In Section \ref{sec2}, we will give the basic definitions and some preliminaries.
Section \ref{sec3} contains the formulation of the problem and connects it with a generalization of an optimization problem introduced by Pascoletti and Serafini \cite{pase84}.
Moreover, relationships to problems with an altered feasible point set are proved.
In Section \ref{sec4}, the existence of optimal solutions and properties of the solution set are studied.
Section \ref{sec5} deals with the question for which varying parameters the optimization problems have the same solution set. Finally, further interdependencies between solution sets for different parameters are shown in Section \ref{sec6}.  

\section{Preliminaries}\label{sec2}

Throughout this paper, $Y$ is assumed to be a real topological vector space.

From now on, $\mathbb{R}$ and $\mathbb{N}$ will denote the set of real numbers and of nonnegative integers, respectively.
We define 
$\mathbb{N}_{>}$ as the set of positive integers,
$\mathbb{R}_{+}:=\{x\in\mathbb{R} \mid x\geq 0\}$, $\mathbb{R}_{>}:=\{x\in\mathbb{R} \mid x > 0\}$,
$\mathbb{R}_{+}^n:=\{(x_1,\ldots ,x_n)^T\in\mathbb{R}^n\mid x_i\geq 0 \,\forall i\in\{1,\ldots, n\}\}$ for each $n\in\mathbb{N}_>$.  $\overline{\mathbb{R}}:=\mathbb{R}\cup \{-\infty ,+\infty\}$ denotes the extended real-valued set.
A set $C$ in $Y$ is said to be a cone iff $\lambda c\in C \mbox{ for all } \lambda\in\mathbb{R}_{+}, c\in C.$ The cone $C$ is called nontrivial iff $C\not=\emptyset$, $C\not=\{0\}$ and $C\not= Y$ hold. It is said to be pointed iff $C\cap (-C)=\{0\} .$ 
Let $A$ be a subset of $Y$.
$0^+A:=\{u\in Y   \mid  A+\mathbb{R}_{+}u\subseteq A\}$ denotes the recession cone of $A$.
$\operatorname*{{core}}A$ stands for the algebraic interior of $A$. 
$\operatorname*{cl}A$, $\operatorname*{int}A$, $\operatorname*{bd}A$ and $\operatorname*{conv}A$ denote the closure, the interior, the boundary and the convex hull, respectively, of $A$.
Given some set $B\subseteq\mathbb{R}$, $d\in Y$, and $D\subseteq Y$, we write $B d:=\{b \cdot d \mid b\in B\}$ and $B D:=\{b \cdot d \mid b\in B, d\in D\}$.
Consider a functional $\varphi :Y\to \overline{\mathbb{R}}$. We use its effective domain 
$\operatorname*{dom}\varphi :=\{y\in Y \mid \varphi(y)\in\mathbb{R}\cup \{-\infty\}\}$.
For some binary relation $\mathcal{R}$ given on $\overline{\mathbb{R}}$, the sets $\operatorname*{lev}_{\varphi,\mathcal{R}}(t):= \{y\in Y  \mid \varphi (y) \mathcal{R} t\}$ are defined for $t\in\mathbb{R}$. In this way, the sublevel sets of $\varphi $ are given as $\operatorname*{lev}_{\varphi,\leq}(t)$.
$\varphi$ is said to be finite-valued on $A$ iff it attains only real values on $A$. It is called finite-valued iff it is finite-valued on $Y$. $\varphi$ is said to be proper iff $\operatorname*{dom}\varphi\not=\emptyset$ and $\varphi$ is finite-valued on $\operatorname*{dom}\varphi .$ According to the rules of convex analysis, $\operatorname*{inf}\emptyset =+\infty .$

From now on, we suppose in this section\\
$(H1_{A,k})$: $A$ is a proper closed subset of $\,Y$ and $k\in -0^+A\setminus \{0\}$.
 
\begin{definition}
The function $\varphi _{A,k}:Y\rightarrow \overline{{\mathbb{R}}}$ defined by
\begin{equation*}
\varphi_{A,k} (y):= \inf \{t\in
{\mathbb{R}} \mid y\in A+tk\} 
\end{equation*}
is called a Gerstewitz functional.
\end{definition}

We will use the following statements from \cite{Wei17a} and \cite{Wei17b}.

\begin{lemma}
\label{t251}
$\varphi _{A,k}$ is lower semicontinuous on $\operatorname*{dom}\varphi _{A,k}$,
\begin{eqnarray}
\operatorname*{dom}\varphi _{A,k} & = & A+\mathbb{R}k,\nonumber\\
\operatorname*{lev}\nolimits_{\varphi_{A,k},\leq}(t)& = & A+t k \quad \forall\, t \in {\mathbb{R}},\label{f-r252n}\\
\operatorname*{lev}\nolimits_{\varphi_{A,k},=}(t) & \subseteq & \operatorname*{bd}A + t k \quad \forall\, t \in\mathbb{R}.\label{r0} 
\end{eqnarray} 
Moreover,
\begin{itemize}
\item[(a)] $\varphi _{A,k}(y)=-\infty \iff y+\mathbb{R}k\subseteq A$.
\item[(b)] $\varphi _{A,k}$ is finite-valued on $\operatorname*{dom}\varphi _{A,k}\setminus A$. 
\item[(c)] $\varphi _{A,k} $ is proper if and only if
$A$ does not contain lines parallel to $k$, i.e.,
\begin{equation*}
\forall\, a\in A:\; a+\mathbb{R}k\not\subseteq A.
\end{equation*}
\item[(d)]If $k\in -\operatorname*{core}0^+A$, then $\varphi _{A, k}$ is finite-valued.
\item[(e)] If $k\in (-0^+A)\cap 0^+A$, then $\varphi_{A,k}$ does not attain any real value.
\item[(f)] If $k\in (-0^+A)\setminus 0^+A$ and $A$ is convex, then $\varphi_{A,k}$ is proper.
\end{itemize}
\end{lemma}

\begin{lemma}\label{l-min-bd}
We have
\begin{equation*}
\varphi  _{A,k}(y)= \min \{t\in {\mathbb{R}} \mid y\in A+tk \}= \min \{t\in {\mathbb{R}} \mid y\in \operatorname*{bd}A+tk \} 
\end{equation*}
for each $y\in Y$ with $\varphi  _{A,k}(y)\in\mathbb{R}$.
\end{lemma}
\begin{proof}
$\varphi  _{A,k}(y)\in\mathbb{R}$ implies
$\varphi  _{A,k}(y)= \min \{t\in {\mathbb{R}} \mid \varphi  _{A,k}(y)\leq t \}= \min \{t\in {\mathbb{R}} \mid y\in A+tk \}$ by (\ref{f-r252n}). Thus, we get the first equation. The second one results from (\ref{r0}).
\end{proof}

\cite[Theorem 3.1]{Wei17a} contains the following results.

\begin{lemma}\label{prop-funcII}
Assume $A-\mathbb{R}_{>} k\subseteq \operatorname*{int} \;A$. 
Then 
\begin{eqnarray}
\varphi _{A,k}  & \mbox{ is } & \mbox{ continuous on } \operatorname*{dom}\varphi _{A,k},\nonumber\\
\operatorname*{dom}\varphi _{A,k} & = &{\mathbb{R}}k+\operatorname*{int}A\not=  \emptyset \mbox{ is an open set},\nonumber\\
\operatorname*{lev}\nolimits _{\varphi_{A,k},<}(t) & = & \operatorname*{int}A + t k\quad\forall\, t \in {\mathbb{R}.} \label{int_less} 
\end{eqnarray} 
Moreover,
\begin{itemize}
\item[(a)] $\varphi _{A,k}(y)=-\infty \iff y+\mathbb{R}k\subseteq \operatorname*{int} \;A$.
\item[(b)] $\mathop{\rm bd}A+\mathbb{R}k$ is the subset of $\,Y$ on which $\varphi _{A,k}$ is finite-valued, and\\
$\operatorname*{dom}\varphi _{A,k}\setminus \operatorname*{int}A\subseteq \mathop{\rm bd}A+\mathbb{R}_{+}k$. 
\item[(c)] The following conditions are equivalent:
\begin{eqnarray*}
\varphi _{A,k} & \mbox{ is } & \mbox{ proper},\\
\operatorname*{int}A  & =  & \operatorname*{bd}A-\mathbb{R}_{>} k.
\end{eqnarray*}
\item[(d)] $\varphi _{A,k} $ is finite-valued if and only if
$Y=\operatorname*{bd}A+\mathbb{R}k$. 
\end{itemize}
\end{lemma}

We will use two other lemmata from \cite{Wei17b}.

\begin{lemma}\label{t-scale}
Consider some arbitrary $\lambda\in\mathbb{R}_{>}$.
Then $(H1_{A,\lambda k})$ holds, $\operatorname*{dom}\varphi _{A,\lambda k}=\operatorname*{dom}\varphi _{A,k}$ and
$\varphi_{A,\lambda k}(y)=\frac{1}{\lambda} \varphi_{A,k}(y)$ for all $y\in Y$.
\end{lemma}

\begin{lemma}\label{0-shift}
Consider some arbitrary $c\in \mathbb{R}$.
Then $(H1_{A+ck,k})$ holds, $\operatorname*{dom}\varphi _{A+ck,k}=\operatorname*{dom}\varphi _{A,k}$ and
$\varphi_{A+ck,k}(y)=\varphi_{A,k}(y)-c$ for all $y\in Y$.
\end{lemma}

In \cite{Wei16b} and \cite{Wei16c}, Gerstewitz functionals are used for the scalarization of vector optimization problems. There, a vector optimization problem is given by a function $f:S\rightarrow Y$ defined on a nonempty set $S$ and by a set $D\subset Y$ which defines the solution concept. $D$ is called the domination set of the problem. A solution of the vector optimization problem is each $s\in S$ for which $f(s)$ is an efficient element of $F:=f(S)$ with respect to $D$. 
An element $y^{0}\in F$ is called an efficient element of $F$
w.r.t. $D$ iff 
\[
F\cap(y^0-D)\subseteq \{y^0\}.
\]
We denote the set of efficient elements of $F$ with respect to $D$ by $\operatorname*{Eff}(F,D).$

In our paper, we need the following statement \cite[Lemma 11]{Wei16c}.

\begin{lemma}\label{l-Eff-prop}
Assume $F, D\subseteq Y$ and $F\subseteq \tilde{F}\subseteq F+(D\cup\{0\})$.
Then $\operatorname*{Eff}(\tilde{F},D)\subseteq \operatorname*{Eff}(F,D)$.
If, additionally, $D\cap (-D)\subseteq \{0\}$ and $D+D\subseteq D$, then $\operatorname*{Eff}(\tilde{F},D)=\operatorname*{Eff}(F,D)$.
\end{lemma}

Furthermore, the following lemmata will be used in the proofs of the next sections.
We get from \cite[Theorem 2.43]{AliBor06}:

\begin{lemma}\label{p-lsc-min}
A real-valued lower semicontinuous function on a nonempty compact subset of a topological space attains a minimum value, and the nonempty set of minimizers is compact. 
\end{lemma}

This implies the next statement by the use of the topology which is induced on a subset of the space.

\begin{lemma}\label{p-minex}
Assume that the following conditions hold:
\begin{itemize}
\item[(a)] $\varphi : Y\to\overline{\mathbb{R}}$ is finite-valued on $M\subseteq \operatorname*{dom}\varphi$ and lower semicontinuous,
\item[(b)] $M$ is a nonempty compact set. 
\end{itemize}
Then the set of minimizers of $\varphi $ on $M$ is nonempty and compact.
\end{lemma}

The next lemma is due to \cite{Wei90}.

\begin{lemma}\label{l-cone-genY}
Assume that $D\subseteq Y$ is a cone with $\operatorname*{int}D\not=\emptyset$ and $d\in\operatorname*{int}D$. Then 
\begin{equation*}
Y=\operatorname*{int}D-\mathbb{R}_>d.
\end{equation*} 
\end{lemma}
\begin{proof}
$V:=\operatorname*{int}D-d$ is a neighbourhood of $0$. $\Rightarrow tV=t(\operatorname*{int}D-d)\subseteq
\operatorname*{int}D-td \;\forall t\in\mathbb{R}_>$ since $D$ is a cone. $\Rightarrow \mathbb{R}_> V\subseteq \operatorname*{int}D-\mathbb{R}_> d$. $\Rightarrow Y=\mathbb{R}_+ V\subseteq\operatorname*{int}D-\mathbb{R}_> d$ since $d\in \operatorname*{int}D$.
\end{proof}

We will also need the following lemmata, which were proved in \cite{weid93}.
\begin{lemma}\label{hab-s21_16}
Assume that $M$ and $D$ are closed convex sets in $\mathbb{R}^{\ell}$. If there exists some $a\in\mathbb{R}^{\ell}$ such that $M\cap (a-D)$ is not bounded, then $0^+(M\cap (b-D))\not=\{0\}$ for each 
$b\in\mathbb{R}^{\ell}$ with $M\cap (b-D)\not=\emptyset$.
\end{lemma}

\begin{lemma}\label{hab-s21_17}
Assume that $D\subseteq \mathbb{R}^{\ell}$ is a closed convex set with $0\in D$ and that $M\subset \mathbb{R}^{\ell}$ is a set for which there exist a polyhedral cone $C\subset \mathbb{R}^{\ell}$ and some $u\in
\mathbb{R}^{\ell}$ such that $-D\setminus\{0\}\subset \operatorname*{int}C$ and $(M-u)\cap
\operatorname*{int}C=\emptyset$.
Then each set $M\cap (b-D)$ with $b\in\mathbb{R}^{\ell}$ is bounded.
\end{lemma}

\section{Problem formulation and related problems}\label{sec3}

From now on, we assume
\begin{eqnarray*}
\mbox{(H1--OP}_{F,a,H,k}\mbox{):} & & F\subseteq Y, a\in Y\mbox{, and }\\
& & H \mbox{ is a proper closed subset of } Y \mbox{ with } k\in 0^+H\setminus\{0\},
\end{eqnarray*}
unless not stated otherwise.

In this paper, we study the optimization problem
\begin{equation}\label{OP_varphi}
\operatorname*{min}_{y\in F}\varphi _{a-H,k}(y).
\end{equation}

Since $\varphi _{a-H,k}$ has been defined as an extended real-valued functional
and $F$ is not necessarily contained in the effective domain of $\varphi _{a-H,k}$, the feasible range $F\cap\operatorname*{dom}\varphi _{a-H,k}$ of the optimization problem does not always coincide with $F$.

\begin{theorem}\label{p-basis-PaHk}
\begin{itemize}
\item[]
\item[(a)] The optimization problem {\rm (\ref{OP_varphi})}
is equivalent to the optimization problem
\begin{eqnarray*}
(P_{F,a,H,k}):& &t\to \min\\
			& &y\in a-H+tk \\
			& &y\in F,\, t\in \mathbb{R},
\end{eqnarray*}
i.e., both problems have the same feasible point set $F\cap\operatorname*{dom}\varphi _{a-H,k}$, each problem has an optimal solution if and only if the other one has an optimal solution, and the optimal solutions $y$ as well as the optimal value of both problems coincide.
\item[(b)] The optimization problem {\rm (\ref{OP_varphi})} has an optimal solution if and only if the optimization problem
\begin{eqnarray*}
(P_{F,a,\operatorname*{bd}H,k}):& &t\to \min\\
			& &y\in a-\operatorname*{bd}H+tk \\
			& &y\in F,\, t\in \mathbb{R},
\end{eqnarray*}
has an optimal solution. In this case, the optimal solutions $y$ as well as the optimal value of both problems coincide. 
\item[(c)] If $(P_{F,a,H,k})$ has a feasible solution $(y^0,t_0)\in F\times \mathbb{R}$, then the set of minimizers $y$ of $(P_{F,a,H,k})$ and its optimal value coincide with those of the problem
\begin{equation}\label{OP-t0}
\operatorname*{min}_{y\in F\cap (a-H+t_0k)}\varphi _{a-H,k}(y).
\end{equation}
\end{itemize}
\end{theorem}
\begin{proof}
By Lemma \ref{l-min-bd}, $\varphi _{a-H,k}(y)\in\mathbb{R}$ implies $\varphi  _{a-H,k}(y)= \min \{t\in {\mathbb{R}} \mid y\in a-H+tk \}= \min \{t\in {\mathbb{R}} \mid y\in a-\operatorname*{bd}H+tk \}$. Thus, we get (a) and (b).\\
Assume now that $(P_{F,a,H,k})$ has a feasible solution $(y^0,t_0)\in F\times \mathbb{R}$.
Since $k\in0^+H$,  $a-H+tk\subseteq a-H+t_0 k\quad\forall t\leq t_0$.
Hence, $(P_{F,a,H,k})$ has the same optimal solutions and the same optimal value as the problem (\ref{OP-t0}).
\end{proof}

Geometrical interpretation of $(P_{F,a,H,k})$:

Stick the set $-H$ to the point $a$ and shift the set $a-H$ along the line $\{ a+tk\mid t\in \mathbb{R}\}$ until the smallest
$t$ is reached for which the intersection of the set with $F$ is not empty,
i.e., for which $F\cap (a-H+tk)\neq \emptyset $.
Then $t$ is the optimal value of $(P_{F,a,H,k})$, and the set of points in which $F$ and $a-H+tk$ intersect is the set of optimal solutions $y$ of $(P_{F,a,H,k})$.

\medskip
We illustrate this by an example.
\begin{example}\label{hab-ex611}
Choose $Y=\mathbb{R}^2$, $F=\{ (y_1,y_2)^T\in \mathbb{R}^2\mid 0\leq y_2\leq y_1\leq 1\}$,
$H=\mathbb{R}^2_+$, $a=(-1,0)^T$, $k=(1,1)^T$.
$(P_{F,a,H,k})$ attains an optimal value $t^{opt}$. $y=(0,0)^T$ is the unique optimal solution of 
$(P_{F,a,H,k})$.
\end{example}

In applications, we will mainly work with $(P_{F,a,H,k})$, where the equivalence to $(P_{F,a,\operatorname*{bd}H,k})$ can be used for restricting the attention in certain steps to the boundary of $H$. The equivalence to problem (\ref{OP_varphi}) and problem (\ref{OP-t0}) is useful for proving properties of $(P_{F,a,H,k})$. Scalarization results for solution sets of vector optimization problems have been formulated in \cite{Wei16b} and \cite{Wei16c} using problem (\ref{OP_varphi}).

From now on, let $M_{F,a,H,k}$ denote the set of optimal solutions $y$ of problem $(P_{F,a,H,k})$. As we have just shown, this set coincides with $\operatorname*{argmin}\limits_{y\in F}\varphi _{a-H,k}(y)$ under assumption (H1--OP$_{F,a,H,k}$).

Sometimes, $F+H$ may be closed or convex though $F$ does not have this property.
In these cases, the following proposition can be useful. Here, we also consider the efficient elements of the minimizer set. As pointed out in \cite{Wei16b} and \cite{Wei16c}, only these minimizers are efficient elements of a vector optimization problem with respect to the same domination set.

\begin{proposition}\label{hab-l615}
Assume, additionally, $0\in H$ and $H+H\subseteq H$.
We get:
\begin{itemize}
\item[(a)] $(P_{F,a,H,k})$ has a solution if and only if $(P_{F+H,a,H,k})$ has a solution. In this case, both optimal values coincide. Moreover, 
\[M_{F,a,H,k}\subseteq M_{F+H,a,H,k}\subseteq M_{F,a,H,k}+H.\]
\item[(b)] If $F+H$ is closed, then 
\begin{eqnarray*}
\operatorname*{cl}M_{F,a,H,k} & \subseteq & M_{F,a,H,k}+H \mbox{ and }\\
\operatorname*{Eff}(\operatorname*{cl}M_{F,a,H,k},H) & \subseteq & \operatorname*{Eff}(M_{F,a,H,k},H).
\end{eqnarray*} 
If, moreover, $H\cap (-H)=\{0\}$ holds, then \[\operatorname*{Eff}(\operatorname*{cl}M_{F,a,H,k},H)=\operatorname*{Eff}(M_{F,a,H,k},H).\]
\end{itemize}
\end{proposition}

\begin{proof}
\begin{itemize}
\item[]
\item[(a)] First, assume that $\varphi _{a-H,k}$ attains on $F$ the minimal value $t$. Then $\varphi _{a-H,k}$ attains the function value $t$ also on $F+H$.\\
Supposition: $\exists\, \bar{t} <t,\, y\in F,\, h\in H:\; \varphi _{a-H,k}(y+h)=\bar{t}$.\\
$\Rightarrow y+h\in a-H+\bar{t} k$ because of (\ref{f-r252n}). $\Rightarrow y\in a-h-H+\bar{t} k
\subseteq a-H+\bar{t} k$.
Then (\ref{f-r252n}) implies $\varphi _{a-H,k}(y)\leq \bar{t}$, a contradiction to the definition
of $t$.
Hence, $t$ is also the minimal value of $\varphi _{a-H,k}$ on $F+H$.
This implies the first inclusion.\\
Assume now that $t$ is the minimum of $\varphi _{a-H,k}$ on $F+H$.
Then $\varphi _{a-H,k}$ does not attain any smaller value than $t$ on $F$.
Take any $y\in F$ and $h\in H$ with $\varphi _{a-H,k}(y+h)=t$.
$\Rightarrow y+h\in a-H+tk$.
$\Rightarrow y\in a-H+tk$.
$\Rightarrow \varphi _{a-H,k}(y)=t$.
$y$ is a minimal solution of $\varphi _{a-H,k}$ on $F$. This yields the second inclusion.
\item[(b)] For $M_{F,a,H,k}=\emptyset$, the statement is obvious.
Assume now $M_{F,a,H,k}\neq \emptyset $ and
$t$ to be the minimum of $\varphi _{a-H,k}$ on $F$.
$\Rightarrow M_{F,a,H,k}=F\cap (a-H+tk)$.
Consider an arbitrary $\bar{y}\in \operatorname*{cl}M_{F,a,H,k}$.
$\operatorname*{cl}M_{F,a,H,k}\subseteq \operatorname*{cl}F\subseteq \operatorname*{cl}(F+H)=F+H$.
$\Rightarrow \exists y\in F,\, h\in H:\; \bar{y}=y+h$.
$\operatorname*{cl}M_{F,a,H,k}\subseteq a-H+tk.\Rightarrow y+h\in a-H+tk$.
$\Rightarrow y\in a-h-H+tk\subseteq a-H+tk.\Rightarrow
y\in M_{F,a,H,k}$.
Thus, $\bar{y}\in M_{F,a,H,k}+H$.
Hence, $\operatorname*{cl}M_{F,a,H,k}\subseteq M_{F,a,H,k}+H$, and Lemma \ref{l-Eff-prop} yields the assertion.
\end{itemize}
\end{proof}

\section{Existence of optimal solutions and properties of the solution set}\label{sec4}

\begin{proposition}\label{hab-s614}
\begin{itemize}
\item[]
\item[(a)] If $F$ is closed, then $M_{F,a,H,k}$ is closed.
\item[(b)] If $F$ and $H$ are convex, then $M_{F,a,H,k}$ as well as the feasible range of $(P_{F,a,H,k})$
are convex.
\end{itemize}
\end{proposition}

\begin{proof}
If $M_{F,a,H,k}=\emptyset$, then the statement of the proposition is fulfilled.
Otherwise, $M_{F,a,H,k}=F\cap(a-H+t_0k)$, where $t_0$ denotes the optimal value of $(P_{F,a,H,k})$. This yields the assertion.
\end{proof}

Let us now investigate the existence of optimal solutions for the considered optimization problems and the compactness of the solution set.

Immediately from $(P_{F,a,H,k})$, we get necessary conditions for the existence of optimal solutions.
\begin{proposition}\label{optlsgn-notw}
If $M_{F,a,H,k}\not=\emptyset$, the following conditions are fulfilled:
\begin{eqnarray}
F\cap (a-H+\mathbb{R}k) & \not= & \emptyset,\label{notw1}\\
\exists\,t\in\mathbb{R}:\, F\cap (a-H+tk) & = & \emptyset\label{notw2}.
\end{eqnarray}
\end{proposition}

Condition (\ref{notw1}) expresses the existence of a feasible solution and is equivalent to $F\cap\operatorname*{dom}\varphi _{a-H,k}\not=\emptyset$. It holds, if $F$ is nonempty and $\varphi _{a-H,k}$ is finite-valued on $F$.

Condition (\ref{notw2}) is fulfilled if and only if $\varphi _{a-H,k}$ is bounded below on $F$. In this case, $\varphi _{a-H,k}$ is finite-valued on $F\cap\operatorname*{dom}\varphi _{a-H,k}$.

\begin{lemma}
The following conditions are equivalent to each other:
\begin{itemize}
\item[(a)] $\varphi _{a-H,k}$ is finite-valued on $F\cap\operatorname*{dom}\varphi _{a-H,k}$.
\item[(b)] $\varphi _{a-H,k}$ is finite-valued on $F\cap (a-H+tk)$ for some $t\in\mathbb{R}$.
\end{itemize}
\end{lemma}

\begin{proof}
Obviously, (a) implies (b). Assume now that there exists some $y\in F\cap\operatorname*{dom}\varphi _{a-H,k}$ with $\varphi _{a-H,k}(y)=-\infty$.
Then $y+\mathbb{R}k\subseteq a-H$ by Lemma \ref{t251}. $\Rightarrow \forall t\in\mathbb{R}:\, y-tk\in a-H$, hence, $y\in a-H+tk$. Then (b) is not fulfilled.
\end{proof}

We will now study sufficient conditions for the existence of optimal solutions, based on the following theorem.
\begin{theorem}\label{optlsgn-hinr}
Assume  
\begin{equation*}
\exists\,t_1\in\mathbb{R}:\; F\cap (a-H+t_1k) \mbox{is nonempty and compact}.
\end{equation*}
\begin{itemize}
\item[(a)] If {\rm (\ref{notw2})} holds, then $M_{F,a,H,k}$ is nonempty and compact.
\item[(b)] Assumption {\rm (\ref{notw2})} holds if and only if 
$\varphi _{a-H,k}$ is finite-valued on\linebreak
$F\cap\operatorname*{dom}\varphi _{a-H,k}$.
\end{itemize}
\end{theorem}

\begin{proof}
(\ref{notw2}) implies by (\ref{f-r252n}) that $\varphi _{a-H,k}$ is finite-valued on $F\cap\operatorname*{dom}\varphi _{a-H,k}$.\\
Assume now that $\varphi _{a-H,k}$ is finite-valued on $F\cap\operatorname*{dom}\varphi _{a-H,k}$.
There exists some $t_1\in\mathbb{R}$ such that $B:=F\cap (a-H+t_1k)$ is nonempty and compact. By Theorem \ref{p-basis-PaHk}, $M_{F,a,H,k}=M_{B,a,H,k}$. Lemma \ref{p-minex} yields that $M_{B,a,H,k}$ is nonempty and compact. This proves (a) and, because of (\ref{f-r252n}), also (b).
\end{proof}

Lemma \ref{p-lsc-min} implies by Lemma \ref{t251}:

\begin{proposition}\label{hab-s613a}
$M_{F,a,H,k}$ is nonempty and compact if
$F$ is a nonempty compact set and $\varphi _{a-H,k}$ is finite-valued.
\end{proposition}

We get especially by Lemma \ref{t251}(d):
 
\begin{corollary}\label{c-core-minex}
Assume that $F$ is a nonempty compact set, and $k\in \operatorname*{core}0^+H$.
Then $M_{F,a,H,k}$ is nonempty and compact.
\end{corollary}

Let us now investigate the set of minimizers for the case that $\varphi _{a-H,k}$ is not necessarily finite-valued.

\begin{proposition}\label{hab-s613b}
Assume that $F$ is a compact set and that the problem $(P_{F,a,H,k})$ has a feasible solution.\\
Then $M_{F,a,H,k}$ is nonempty and compact under each of the following conditions:
\begin{itemize}
\item[(a)] $F\cap (a-H)=\emptyset$,
\item[(b)] $F\cap (a-\operatorname*{int}H)=\emptyset$ and $H+\mathbb{R}_{>}k\subset \operatorname*{int}H$,
\item[(c)] $H$ does not contain a line in direction $k$,
\item[(d)] $H$ is convex and $k\not\in -0^+H$.
\end{itemize}
\end{proposition}

\begin{proof}
$(P_{F,a,H,k})$ has a feasible solution. $\Rightarrow \exists t_0\in\mathbb{R}:\; B:=F\cap (a-H+t_0 k)\not=\emptyset$. By Theorem \ref{p-basis-PaHk}, $M_{F,a,H,k}=M_{B,a,H,k}$.
$B$ is compact since $F$ is compact and $H$ is closed.
$\varphi _{a-H,k}$ is finite-valued on $B$
in the cases (a) and (c) by Lemma \ref{t251}(b) and (c), respectively, in case (b) by Lemma \ref{prop-funcII}(b), in case (d) by Lemma \ref{t251}(f).
Lemma \ref{p-minex} yields the assertion.
\end{proof}

For $Y=\mathbb{R}^{\ell}$, the existence of optimal solutions of $(P_{F,a,H,k})$ can be guaranteed without the assumption that $F$ is compact.

\begin{proposition}\label{hab-s613c1}
Assume that the following conditions are fulfilled:
\begin{itemize}
\item[(a)] $a\in \mathbb{R}^{\ell}$, $F\subseteq \mathbb{R}^{\ell}$ is a nonempty closed set, and there exists some $u\in \mathbb{R}^{\ell}$ with $F\subseteq u+\mathbb{R}^{\ell}_+$.
\item[(b)] $H$ is a proper closed subset of $\mathbb{R}^{\ell}$ with $\mathbb{R}^{\ell}_+\subseteq 0^+H$ and $k\in \operatorname*{core}0^+H$ that, for each $j\in\{1,\ldots,\ell\}$, does not contain any line in the direction $\mathfrak{e}^j$.
\end{itemize}
Then $M_{F,a,H,k}$ is nonempty and compact.
\end{proposition}

\begin{proof}
By Lemma \ref{t251}(d), $\varphi_{a-H,k}$ is finite-valued. Hence,
$(P_{F,a,H,k})$ has a feasible solution. $\Rightarrow \exists t_0\in\mathbb{R}:\; B:=F\cap (a-H+t_0 k)\not=\emptyset$.\\ 
For each $j\in\{1,\ldots,\ell\}$, $H$ does not contain any line in the direction $\mathfrak{e}^j$. Thus, $\forall j\in\{1,\ldots,\ell\}\,\exists s_j\in\mathbb{R}$:
\begin{equation}\label{notej}
(-u+a+t_0k)-s_j\mathfrak{e}^j\not\in H.
\end{equation}
$z_j:=u_j+s_j\quad\forall j\in\{1,\ldots,\ell\}.$ Suppose that there exists some $y\in B$ and some $i\in\{1,\ldots,\ell\}$ with $y_i > z_i$.
$y\in F\subseteq u+\mathbb{R}^{\ell}_+$ implies $u+s_i\mathfrak{e}^i\in y-\mathbb{R}^{\ell}_+\subseteq (a-H+t_0 k)-\mathbb{R}^{\ell}_+\subseteq a-H+t_0 k$, a contradiction to (\ref{notej}).
Thus, $B\subseteq z-\mathbb{R}^{\ell}_+$. Hence, $B$ is compact. 
By Corollary \ref{c-core-minex}, $M_{B,a,H,k}$ is nonempty and compact.
Theorem \ref{p-basis-PaHk} yields the assertion.
\end{proof}

Note that $\mathbb{R}^{\ell}_+\subseteq 0^+H$ and $k\in \operatorname*{int}\mathbb{R}^{\ell}_+$ imply $k\in \operatorname*{core}0^+H$.

The assumption in Proposition \ref{hab-s613c1} that refers to lines in directions $\mathfrak{e}^j$ is not superfluous, even for convex sets $H$.
\begin{example}\label{hab-ex613}
$Y=\mathbb{R}^2$, $F=\{ (y_1,y_2)^T\in \mathbb{R}^2\mid y_1>0,\: y_2=\frac{1}{y_1}\}$,
$a=(0,0)^T$, $k=(1,1)^T$
and $H=\{ (y_1,y_2)^T\in \mathbb{R}^2\mid y_1\geq 0\}$ fulfill all assumptions of Proposition \ref{hab-s613c1} but the condition that $H$ does not contain any line in direction $(0,1)^T$.
$\,F\cap (a-H+tk)=\emptyset \quad \forall t\leq 0$ and
$F\cap (a-H+tk)\neq \emptyset \quad \forall t>0$.
Thus, $(P_{F,a,H,k})$ does not have any optimal solution.
\end{example}

Furthermore, the statement of Proposition \ref{hab-s613c1} is not true any more without the assumption that $F$ is bounded below.

\begin{example}\label{hab-ex614}
Consider $Y=\mathbb{R}^2$, $F=\{ (y_1,y_2)^T\in \mathbb{R}^2\mid y_2=0\}$,
$H=\{ (y_1,y_2)^T\in\mathbb{R}^2\mid y_1>0,\: y_2\ge \frac{1}{y_1}\}$,
$k=(1,1)^T$.
For each $a\in \mathbb{R}^2$, $(P_{F,a,H,k})$ does not have any optimal solution.
\end{example}

Proposition \ref{hab-s613c1} implies a special statement for convex cones $H$, which results from the following lemma.

\begin{lemma}\label{l-hab-s613c1_cone}
Assume that $H\subset \mathbb{R}^{\ell}$ is a nontrivial closed pointed convex cone with $\mathbb{R}^{\ell}_+\subseteq H$.
Then, for each $j\in\{1,\ldots,\ell\}$, $H$ does not contain any line in the direction $\mathfrak{e}^j$. 
\end{lemma}

\begin{proof}
Consider an arbitrary $j\in\{1,\ldots,\ell\}$. $\mathfrak{e}^j\in H$ by $\mathbb{R}^{\ell}_+\subseteq H$, and $\mathfrak{e}^j\in H\setminus (-H)$ since $H$ is pointed.
By Lemma \ref{t251}(f), $\varphi_{-H,\mathfrak{e}^j}$ is proper. Thus, by Lemma \ref{t251}(c), $H$ does not contain any line in direction $\mathfrak{e}^j$.
\end{proof}

Thus, we get by Proposition \ref{hab-s613c1}:

\begin{corollary}\label{hab-s613c1_cone}
Assume that the following conditions are fulfilled:
\begin{itemize}
\item[(a)] $a\in \mathbb{R}^{\ell}$, $F\subseteq \mathbb{R}^{\ell}$ is a nonempty closed set, and there exists some $u\in \mathbb{R}^{\ell}$ with $F\subseteq u+\mathbb{R}^{\ell}_+$.
\item[(b)] $H\subseteq \mathbb{R}^{\ell}$ is a nontrivial closed pointed convex cone with  $\mathbb{R}^{\ell}_+\subseteq H$ and $k\in \operatorname*{int}H$.
\end{itemize}
Then $M_{F,a,H,k}$ is nonempty and compact.
\end{corollary}

Let us now consider the optimization problems on $Y=\mathbb{R}^{\ell}$ for the case that $k$ is not necessarily an element of $\operatorname*{core}0^+H$.

\begin{proposition}\label{hab-s613c2}
Assume that the following conditions hold:
\begin{itemize}
\item[(a)] $a\in \mathbb{R}^{\ell}$, $F\subseteq \mathbb{R}^{\ell}$ is closed, and there exists some $u\in \mathbb{R}^{\ell}$ with $F\subseteq u+\mathbb{R}^{\ell}_+$.
\item[(b)] $H$ is a proper closed subset of $\mathbb{R}^{\ell}$ with $\mathbb{R}^{\ell}_+\subseteq 0^+H$ and $k\in 0^+H\setminus\{0\}$ that, for each $j\in\{1,\ldots,\ell\}$, does not contain any line in the direction $\mathfrak{e}^j$.
\item[(c)] The problem $(P_{F,a,H,k})$ has a feasible solution.
\item[(d)] One of the following assumptions is fulfilled:
\begin{itemize}
\item[(i)] $F\cap (a-H)=\emptyset$,
\item[(ii)] $F\cap (a-\operatorname*{int}H)=\emptyset$ and $H+\mathbb{R}_{>}k\subset \operatorname*{int}H$,
\item[(iii)] $H$ does not contain a line in direction $k$,
\item[(iv)] $H$ is convex and $k\not\in -0^+H$.
\end{itemize}
\end{itemize}
Then $M_{F,a,H,k}$ is nonempty and compact.
\end{proposition}

\begin{proof}
$(P_{F,a,H,k})$ has a feasible solution. $\Rightarrow \exists t_0\in\mathbb{R}:\; B:=F\cap (a-H+t_0 k)\not=\emptyset$. 
By the same arguments as in the proof of Proposition \ref{hab-s613c1}, $B$ is compact. 
By Proposition \ref{hab-s613b}, $M_{B,a,H,k}$ is nonempty and compact.
Theorem \ref{p-basis-PaHk} yields the assertion.
\end{proof}

For convex sets $H$, Proposition \ref{hab-s613c2} implies a statement which can be proved by means of the following lemma.

\begin{lemma}\label{hab-l611}
Assume that $H\subset\mathbb{R}^{\ell}$ is a proper closed convex subset of $\mathbb{R}^{\ell}$ with $H+(\mathbb{R}^{\ell}_+\setminus \{ 0 \} )
\subseteq \operatorname*{int}H$, $H+\mathbb{R}_> k\subseteq \operatorname*{int}H$ and $k\not\in -0^+H$.
Then, for each $j\in\{1,\ldots,\ell\}$, $H$ does not contain any line in direction $\mathfrak{e}^j$.
\end{lemma}

\begin{proof}
Supposition: $\exists i\in \{1,\ldots ,\ell\}:\: H$ contains some line in direction $\mathfrak{e}^i$.\\
$\Rightarrow \exists u\in \mathbb{R}^{\ell}:\, \{ u+s\, \mathfrak{e}^i\mid s\in \mathbb{R}\}
\subseteq H$.
$\Rightarrow \{ u+s\, \mathfrak{e}^i\mid s\in \mathbb{R}\} =\{ u+(s-1)\mathfrak{e}^i\mid s\in \mathbb{R}\}
+\mathfrak{e}^i\subseteq H+(\mathbb{R}_+^{\ell}\setminus \{ 0 \} )
\subseteq \operatorname*{int}H$.\\
By Lemma \ref{t251}(f), $\varphi_{-H,k}$ is proper. Hence, Lemma \ref{prop-funcII}(c) implies $\operatorname*{int}H =\operatorname*{bd}H+\mathbb{R}_>k$.
$u\in \operatorname*{int}H \Rightarrow \exists h\in \operatorname*{bd}H,\,
t_0>0:\: h=u-t_0k$.
$h+\mathfrak{e}^i\in \operatorname*{int}H$ since $H+(\mathbb{R}_+^{\ell}\setminus \{ 0 \} )\subseteq \operatorname*{int}H$.
$\Rightarrow \exists t_1>0:\: h+\mathfrak{e}^i-t_1k\in \operatorname*{int}H$.\\
$t_2 := \frac{t_0}{t_0+t_1}\in (0,1).\quad s_0:= -\frac{t_2}{1-t_2}$.
$\Rightarrow t_2+(1-t_2)s_0=0,\: t_2(t_0+t_1)=t_0$.
$\Rightarrow h=u-t_0k=u+(t_2+(1-t_2)s_0)e^i-t_2(t_0+t_1)k=
t_2(u+e^i-(t_0+t_1)k)+(1-t_2)(u+s_0\mathfrak{e}^i)=t_2(h+\mathfrak{e}^i-t_1k)+(1-t_2)(u+s_0\mathfrak{e}^i)\in \operatorname*{int}H$ since $\operatorname*{int}H$
is convex.
This contradicts $h\in \operatorname*{bd}H$.
\end{proof}

Thus, we get by Proposition \ref{hab-s613c2}:

\begin{corollary}\label{hab-c611b}
Assume that the following conditions hold:
\begin{itemize}
\item[(a)] $a\in \mathbb{R}^{\ell}$, $F\subseteq \mathbb{R}^{\ell}$ is closed and there exists some $u\in \mathbb{R}^{\ell}$ with $F\subseteq u+\mathbb{R}^{\ell}_+$.
\item[(b)] $H$ is a proper closed convex subset of $\mathbb{R}^{\ell}$ with $H+(\mathbb{R}^{\ell}_+\setminus \{ 0 \} )
\subseteq \operatorname*{int}H$, $H+\mathbb{R}_> k\subseteq \operatorname*{int}H$ and $k\not\in -0^+H$.
\item[(c)] The problem $(P_{F,a,H,k})$ has a feasible solution.
\end{itemize}
Then $M_{F,a,H,k}$ is nonempty and compact.
\end{corollary}

Without the assumption $F\subseteq u+\mathbb{R}^{\ell}_+$ in Corollary \ref{hab-c611b}, the existence of optimal solutions for $(P_{F,a,H,k})$ can depend on the choice of $a$,
even if $H$ is a closed convex cone.

\begin{example}\label{hab-ex615}
Assume $Y=\mathbb{R}^3$ and that $H$ is the convex cone generated by the vectors $(2,0,-1)^T$, $(0,2,-1)^T$
and $(-1,0,2)^T$. Choose $k=(1,1,1)^T$, $a=(0,0,0)^T$ and $b=(1,-1,-\frac{1}{2})^T$.
The set $F := \{ (y_1,y_2,y_3)^T\in \mathbb{R}^3\mid 2y_1+y_2+2y_3=0,\,
y_1>0,\, y_2\le -\frac{1}{y_1}\} $ is contained in the same hyperplane as the points
$(0,0,0)^T$,$(2,0,-2)^T$ und $(0,-2,1)^T$.
It is closed and convex.
Since $(a-H)\cap F=\emptyset$ and $(a-\operatorname*{int}H+tk)\cap F\neq
\emptyset\; \forall t>0$, $(P_{F,a,H,k})$ does not have any optimal solution.
But there exist optimal solutions of the problem $(P_{F,b,H,k})$. These are
the elements of the set 
$(b-\operatorname*{bd}H)\cap F=(b-H)\cap F=\{ (y_1,y_2,y_3)^T\in \mathbb{R}^3\mid y_1=1,\,
y_2=-2\lambda -1,\, y_3=\lambda -\frac{1}{2},\, \lambda \in \mathbb{R}_+\} $.
\end{example}

Proposition \ref{hab-s613c2} and Lemma \ref{l-hab-s613c1_cone} imply for convex cones $H$:

\begin{corollary}\label{hab-c611a}
Assume that the following conditions hold:
\begin{itemize}
\item[(a)] $a\in \mathbb{R}^{\ell}$, $F\subseteq \mathbb{R}^{\ell}$ is closed, and there exists some $u\in \mathbb{R}^{\ell}$ with $F\subseteq u+\mathbb{R}^{\ell}_+$.
\item[(b)] $H\subset\mathbb{R}^{\ell}$ is a nontrivial closed convex pointed cone with $\mathbb{R}^{\ell}_+\subseteq H$ and $k\in H\setminus (-H)$.
\item[(c)] The problem $(P_{F,a,H,k})$ has a feasible solution.
\end{itemize}
Then $M_{F,a,H,k}$ is nonempty and compact.
\end{corollary}

Without the assumption $F\subseteq u+\mathbb{R}^{\ell}_+$ in Corollary \ref{hab-c611a}, the existence of optimal solutions for $(P_{F,a,H,k})$ can depend on the choice of $a$, even if $H$ is the nonnegative orthant.

\begin{example}\label{hab-ex616}
Consider $Y=\mathbb{R}^3$, $H=\mathbb{R}_+^3$, $k=(1,1,1)^T$, $a=(0,0,0)^T$, $b=(1,-1,-1)^T$.
$F := \{ (y_1,y_2,y_3)^T\in \mathbb{R}^3\mid y_1+y_3=0,\,y_2\le-\frac{1}{y_1},\,
y_1>0\} $ is a closed convex set, which is contained in the hyperplane $\{ (y_1,y_2,y_3)^T\in \mathbb{R}^3\mid y_1+y_3=0\} $.
$(a-H)\cap F=\emptyset $, but $(a-\operatorname*{int}H+tk)\cap F\neq \emptyset \quad \forall t>0$.
Thus, $(P_{F,a,H,k})$ does not have any optimal solution.
But $(b-\operatorname*{bd}H)\cap F=(b-H)\cap F=\{ (y_1,y_2,y_3)^T\in \mathbb{R}^3\mid y_1=1,\,
y_2\le -1,\, y_3=-1\} $
is the set of minimizers $y$ of $(P_{F,b,H,k})$.
\end{example}

We now turn to statements in $Y=\mathbb{R}^{\ell}$ without the assumption that $F$ is bounded below.

\begin{proposition}\label{hab-s613d1}
Assume that the following conditions are fulfilled:
\begin{itemize}
\item[(a)] $a\in \mathbb{R}^{\ell}$, $F\subseteq \mathbb{R}^{\ell}$ is a nonempty closed set. 
\item[(b)] $H$ is a proper closed subset of $\mathbb{R}^{\ell}$ with $k\in \operatorname*{core}0^+H$.
\item[(c)] There exist a polyhedral cone $C\subset \mathbb{R}^{\ell}$, some $z\in\operatorname*{cl}\operatorname*{conv}H$ and $u\in
\mathbb{R}^{\ell}$ such that
\begin{eqnarray*}
z-\operatorname*{cl}\operatorname*{conv}H & \subset & \operatorname*{int}C\cup\{0\}\quad\mbox{ and}\\
(F-u) & \cap & \operatorname*{int}C  =  \emptyset.
\end{eqnarray*}
\end{itemize}
Then $M_{F,a,H,k}$ is nonempty and compact.
\end{proposition}

\begin{proof}
By Lemma \ref{t251}(d), $\varphi_{a-H,k}$ is finite-valued. Hence,
$(P_{F,a,H,k})$ has a feasible solution. $\Rightarrow \exists t_0\in\mathbb{R}:\; B:=F\cap (a-H+t_0 k)\not=\emptyset$.
$\,B\subseteq F\cap (a+t_0k-z+(z-\operatorname*{cl}\operatorname*{conv}H))$ is bounded by Lemma \ref{hab-s21_17}, hence compact. 
By Corollary \ref{c-core-minex}, $M_{B,a,H,k}$ is nonempty and compact.
Theorem \ref{p-basis-PaHk} yields the assertion.
\end{proof}

The assumptions of Proposition \ref{hab-s613d1} do not imply that $F$ is bounded below or that $H$ is a cone .

\begin{example}
The sets $H:=\{ (y_1,y_2)^T\in \mathbb{R}^2\mid y_1\geq -1, y_2\geq -1, (y_1+1)(y_2+1)\geq 1\}$ and $F:=\{ (y_1,y_2)^T\in \mathbb{R}^2\mid y_1\geq 0, y_2\geq -\frac{y_1}{2}\}$ fulfill
the conditions in Proposition \ref{hab-s613d1} with $C:=\{ (y_1,y_2)^T\in \mathbb{R}^2\mid y_1+y_2\leq 0\}$ and $u=z=(0,0)^T$. 
\end{example}

For the case that $k$ is not necessarily an element of $\operatorname*{core}0^+H$, we can prove the following statement.

\begin{proposition}\label{hab-s613d2}
Assume that the following conditions are satisfied:
\begin{itemize}
\item[(a)] $a\in \mathbb{R}^{\ell}$, $F\subseteq \mathbb{R}^{\ell}$ is a nonempty closed set. 
\item[(b)] $H$ is a proper closed subset of $\mathbb{R}^{\ell}$ with $k\in 0^+H\setminus\{0\}$.
\item[(c)] There exist a polyhedral cone $C\subset \mathbb{R}^{\ell}$, some $z\in\operatorname*{cl}\operatorname*{conv}H$ and $u\in
\mathbb{R}^{\ell}$ such that
\begin{eqnarray*}
z-\operatorname*{cl}\operatorname*{conv}H & \subset & \operatorname*{int}C\cup\{0\}\quad\mbox{ and}\\
(F-u) & \cap & \operatorname*{int}C  =  \emptyset.
\end{eqnarray*}
\item[(d)] The problem $(P_{F,a,H,k})$ has a feasible solution.
\item[(e)] One of the following assumptions is fulfilled:
\begin{itemize}
\item[(i)] $F\cap (a-H)=\emptyset$,
\item[(ii)] $F\cap (a-\operatorname*{int}H)=\emptyset$ and $H+\mathbb{R}_{>}k\subset \operatorname*{int}H$,
\item[(iii)] $H$ does not contain a line in direction $k$.
\end{itemize}
\end{itemize}
Then $M_{F,a,H,k}$ is nonempty and compact.
\end{proposition}

\begin{proof}
$(P_{F,a,H,k})$ has a feasible solution. $\Rightarrow \exists t_0\in\mathbb{R}:\; B:=F\cap (a-H+t_0 k)\not=\emptyset$. 
$\,B\subseteq F\cap (a+t_0k-z+(z-\operatorname*{cl}\operatorname*{conv}H))$ is bounded by Lemma \ref{hab-s21_17}, hence compact. 
By Proposition \ref{hab-s613b}, $M_{B,a,H,k}$ is nonempty and compact.
Theorem \ref{p-basis-PaHk} yields the assertion.
\end{proof}

Apply now the idea of the previous proposition to convex sets $H$.

\begin{proposition}\label{hab-s613d3}
Assume that the following conditions are fulfilled:
\begin{itemize}
\item[(a)] $a\in \mathbb{R}^{\ell}$, $F\subseteq \mathbb{R}^{\ell}$ is a nonempty closed set. 
\item[(b)] $H$ is a proper closed convex subset of $\mathbb{R}^{\ell}$ with $k\in 0^+H\setminus (-0^+H)$.
\item[(c)] There exist a polyhedral cone $C\subset \mathbb{R}^{\ell}$, some $h\in H$ and $u\in
\mathbb{R}^{\ell}$ such that
\begin{eqnarray*}
h-H & \subset & \operatorname*{int}C\cup\{0\}\quad\mbox{ and}\\
(F-u) & \cap & \operatorname*{int}C  =  \emptyset.
\end{eqnarray*}
\item[(d)] The problem $(P_{F,a,H,k})$ has a feasible solution.
\end{itemize}
Then $M_{F,a,H,k}$ is nonempty and compact.
\end{proposition}

\begin{proof}
$(P_{F,a,H,k})$ has a feasible solution. Then, there exists some $t_0\in\mathbb{R}$ with $B:=F\cap (a-H+t_0 k)\not=\emptyset$. 
$\, B=F\cap(a-h+t_0k+(h-H))$ is bounded by Lemma \ref{hab-s21_17}, hence compact. 
By Proposition \ref{hab-s613b}, $M_{B,a,H,k}$ is nonempty and compact.
Theorem \ref{p-basis-PaHk} yields the assertion.
\end{proof}

Condition (d) in Proposition \ref{hab-s613d3} is fulfilled under the assumptions (a)-(c) if $k\in \operatorname*{core}0^+H$.

\section{Parameter control}\label{sec5}

Problems $(P_{F,a,H,k})$ with varying parameters $a$ and $k$ are used in vector optimization procedures.

\begin{proposition}\label{hab-s619}
Assume that $H$ is a proper closed convex subset of $Y$ with $0^+H\not=\{0\}$ and that $F$ is a nonempty subset of $Y$ for which
$F+0^+H$ is convex.\\
Then
$\{ (a,k)\in Y\times (0^+H\setminus \{ 0 \} )\mid
(P_{F,a,H,k})
\mbox{ has a feasible solution}\}$
is a convex set.
\end{proposition}

\begin{proof}
Assume $(a^1,k^1),(a^2,k^2)\in Y\times (0^+H\setminus \{
0 \} )$ such that
$(P_{F,a^1,H,k^1})$ and $(P_{F,a^2,H,k^2})$ have feasible solutions.
$\Rightarrow \exists y^1,y^2\in F,\, t_1,t_2\in \mathbb{R}:\; y^1\in
a^1-H+t_1k^1,
\: y^2\in a^2-H+t_2k^2$.
Take any $\lambda \in (0,1)$.
$\lambda y^1+(1-\lambda )y^2\in \lambda a^1+(1-\lambda )a^2-H+\lambda t_1k^1+(1-\lambda )t_2k^2$
since $H$ is convex.
$t := \max (t_1,t_2)$.
$\Rightarrow h := \lambda
(t-t_1)k^1+(1-\lambda)(t-t_2)k^2\in 0^+H$
since $0^+H$ is a convex cone.
$\lambda t_1k^1+(1-\lambda )t_2k^2 =t(\lambda k^1+(1-\lambda )k^2)-h$ implies
$\lambda y^1+(1-\lambda )y^2\in \lambda a^1+(1-\lambda )a^2-H+t(\lambda k^1+(1-\lambda )k^2)$.
Since $F+0^+H$ is convex, there exist $y^3\in F$ and
$h^1\in 0^+H$ such that
$\lambda y^1+(1-\lambda )y^2=y^3+h^1$.
Hence, $y^3$ is a feasible solution of $(P_{F,\lambda a^1+(1-\lambda )a^2,H,\lambda k^1+(1-\lambda )k^2})$. 
\end{proof}

Note that $F+0^+H$ is convex if $F$ is convex.

Lemma \ref{t251}(e) and Lemma \ref{t-scale} imply restrictions to the set of parameters $k$ which have to be considered. 

\begin{proposition}\label{hab-s61_12}
Assume (H1--OP$_{F,a,H,k}$).
\begin{itemize}
\item[(a)] If $k\in -0^+H$, then $(P_{F,a,H,k})$ does not have an optimal solution.
\item[(b)] For each $\lambda \in\mathbb{R}_>$, the problem $(P_{F,a,H,\lambda k})$ has the same feasible vectors $y$ and the same optimal solutions $y$ as $(P_{F,a,H,k})$.
\end{itemize}
\end{proposition}

The proposition underlines that replacing $k$ by another vector in the same direction does not alter the optimal solutions. Hence, it is sufficient to consider only one vector $k$ per direction, e.g., to restrict $k$ to unit vectors if $Y$ is a normed space. In $Y=\mathbb{R}^{\ell}$, the range for $k$ can also be restricted by Proposition \ref{hab-s61_10}(d).

We now investigate whether the set of parameters $a$ can be restricted.
We get from Lemma \ref{0-shift}:
\begin{proposition}\label{vor-hab-s61_11}
Assume (H1--OP$_{F,a,H,k}$).
Then $M_{F,a,H,k}=M_{F,a+ck,H,k}$ for each $c\in\mathbb{R}$.
\end{proposition}

This implies:

\begin{proposition}\label{hab-s61_11}
If $\Lambda\subset Y$ with $Y=\{\Lambda + \beta k | \beta\in\mathbb{R}\}$, then
\[ \quad\bigcup\limits_{a\in Y}M_{F,a,H,k}=\bigcup\limits_{a\in\Lambda}\;M_{F,a,H,k}.\]
\end{proposition}

$\Lambda$ can be chosen as the complementary space of the linear subspace
$\{\beta k\mid \beta\in\mathbb{R}\}$. Thus, in $Y=\mathbb{R}^{\ell}$, the dimension of the parameter set for $a$ can be reduced by one. 

\begin{proposition}\label{hab-s61_10}
Assume $F\subseteq \mathbb{R}^{\ell}$ and $H$ is a proper closed subset of $\mathbb{R}^{\ell}$ with $0^+H\not=\{0\}$.
\begin{itemize}
\item[(a)] For each $k\in 0^+H\setminus\{0\}$ and $j\in\{1,\ldots,\ell\}$ with
$k_j\neq 0$, we have
\[ \quad
\bigcup\limits_{a\in\mathbb{R}^{\ell}}M_{F,a,H,k}=\bigcup\limits_{a\in\{y\in\mathbb{R}^{\ell}\mid
y_j=0\}}\;M_{F,a,H,k}.\]
\item[(b)] If $k\in 0^+H\setminus\{0\}$ and $k\in \operatorname*{int}\mathbb{R}^{\ell}_+\cup (-\operatorname*{int}\mathbb{R}^{\ell}_+)$, we have
\[  \quad
\bigcup\limits_{a\in\mathbb{R}^{\ell}}M_{F,a,H,k}=\bigcup\limits_{a\in\mathbb{R}^{\ell}_+}M_{F,a,H,k}=
\bigcup\limits_{a\in -\mathbb{R}^{\ell}_+}M_{F,a,H,k}.\]
\item[(c)] If $k\in 0^+H\setminus\{0\}$ with $\sum\limits_{i=1}^{\ell} k_i\neq 0$, then
\[\quad\bigcup\limits_{a\in
\mathbb{R}^{\ell}}M_{F,a,H,k}=\bigcup\limits_{a\in\{y\in\mathbb{R}^{\ell}\mid
\sum\limits_{i=1}^{\ell} y_i=0\}}\;M_{F,a,H,k}.\]
\item[(d)] For $a\in\mathbb{R}^{\ell}$ being fixed, we get
\[ \quad\bigcup\limits_{k\in \{y\in 0^+H\mid
\sum\limits_{i=1}^{\ell} y_i\neq 0\}}\;M_{F,a,H,k}=
\bigcup\limits_{k\in \{y\in 0^+H\mid
\sum\limits_{i=1}^{\ell} y_i=1\}}\;M_{F,a,H,k}.\]
\end{itemize}
\end{proposition}

\begin{proof}
Apply first Proposition \ref{vor-hab-s61_11}. 
\begin{itemize}
\item[(a)] For $c= -\frac{a_j}{k_j} $, we get $(a+ck)_j=0$.
\item[(b)] Choose $n\in\{1,\ldots,\ell\}$ with 
$\frac{a_n}{k_n}=
\max\limits_{i\in\{1,\ldots,\ell\} }\frac{a_i}{k_i}$,
$m\in\{1,\ldots,\ell \}$ with\linebreak
 $\frac{a_m}{k_m}=
\min\limits_{i\in\{1,\ldots,\ell\} }\frac{a_i}{k_i}$.\\
If $k\in int\,\mathbb{R}_+^{\ell}$, we get for $c=-\frac{a_n}{k_n}$
$\quad a+ck\in-\mathbb{R}_+^{\ell}$ and for $c=-\frac{a_m}{k_m}$
$\quad a+ck\in\mathbb{R}_+^{\ell}$.
If $k\in - int\,\mathbb{R}_+^{\ell}$, the two vectors have the opposite sign.
\item[(c)] For $c=-\frac{\sum\limits_{i=1}^{\ell}
a_i}{\sum\limits_{i=1}^{\ell} k_i}$,
$\sum\limits_{i=1}^{\ell} (a+ck)_i=0$ holds.
\item[(d)] $M_{F,a,H,k}=M_{F,a,H,\lambda k}\quad\forall
\lambda\in\mathbb{R}_>$ implies with $\lambda:=\frac{1}{\sum\limits_{i=1}^{\ell} k_i}$
the assertion by Proposition \ref{hab-s61_12}.
\end{itemize}
\end{proof}

\section{Solution sets for varying parameters}\label{sec6}

The solution of a problem $(P_{F,a,H,k})$ can deliver some information about solution sets for problems with altered parameters $a$ and $k$.

\begin{lemma}\label{hab-l612}
Assume (H1--OP$_{F,a,H,k}$) and $k\in \operatorname*{int}0^+H$. Then: 
\begin{itemize}
\item[(a)] Each function $\varphi _{b-H,k^0}$ with $b\in Y$ and $k^0\in\operatorname*{int}0^+H$ is finite-valued.
\item[(b)] If $\varphi _{a-H,k}$ is bounded below on $F$, then $\varphi _{b-H,k^0}$ is bounded below on $F$ for all $b\in Y$ and $k^0\in\operatorname*{int}0^+H$.
\end{itemize}
\end{lemma}

\begin{proof}
\begin{itemize}
\item[]
\item[(a)] results from Lemma \ref{t251}(d).
\item[(b)] $\varphi _{a-H,k}$ is bounded below on $F$.
$\Rightarrow \exists s\in \mathbb{R}\;\forall y\in F:\;\varphi _{a-H,k}(y)\geq s$.
Then (\ref{int_less}) implies $(a-\operatorname*{int}H+sk)\cap F=\emptyset$.
By Lemma \ref{l-cone-genY}, $Y=\bigcup\limits_{\alpha\in \mathbb{R}}(\operatorname*{int}0^+H+\alpha k^0)=
\bigcup\limits_{\alpha\in \mathbb{R}}(a+sk-\operatorname*{int}0^+H-\alpha k^0)$.
$\Rightarrow \exists t\in \mathbb{R}:\; b\in a+sk-\operatorname*{int}0^+H-tk^0$.
$\Rightarrow  b+tk^0\in a+sk-\operatorname*{int}0^+H$.
$\Rightarrow b-\operatorname*{int}H+tk^0\subseteq a+sk-(\operatorname*{int}H+\operatorname*{int}0^+H)\subseteq
a+sk-\operatorname*{int}H$.
$\Rightarrow (b-\operatorname*{int}H+tk^0)\cap F=\emptyset$. (\ref{int_less}) implies 
$\varphi _{b-H,k^0}\geq t\; \forall y\in F$,
i.e., $\varphi _{b-H,k^0}$ is bounded below on $F$.
\end{itemize}
\end{proof}

In Lemma \ref{hab-l612}, it is essential that $k$ and $k^0$ belong to the interior of $0^+H$.

\begin{example}\label{hab-ex617}
$H=\{ (y_1,y_2)^T\in \mathbb{R}^2\mid y_2\ge y_1^2\}$ fulfills the assumptions of Lemma \ref{hab-l612} without the condition that the interior of $0^+H$ is nonempty.
$k=(0,1)^T\in 0^+H$.
Consider $F=\{ (y_1,y_2)^T\in \mathbb{R}^2\mid y_1\ge 0,\, y_2=-y_1^2\} \cup
\{(-1,0)^T\}$.
Then $F\cap -\operatorname*{int}H=\emptyset $, hence, 
$\varphi _{-H,k}(y)\geq 0\; \forall y\in F$,
i.e., $\varphi _{-H,k}$ is bounded below on $F$.
Let $t$ be an arbitrary negative number and $b=(1,0)^T$.
$y_1 := \frac{3}{4}-\frac{t}{2}>0$, $y_2
:= -y_1^2$. Then $y\in F$.\\
$\bar{y} := y-b-tk= (-\frac{t}{2}-\frac{1}{4},
-\frac{t^2}{4}-\frac{t}{4}-\frac{9}{16})^T
\in -\operatorname*{int}H$ since $\bar{y}_2<-\bar{y}_1^2$.
Hence, $y\in b-\operatorname*{int}H+tk$ and $\varphi _{b-H,k}(y)<t$.
Since $t$ can be arbitrarily small, $\varphi _{b-H,k}$ is not bounded below on $F$.
\end{example}

Lemma \ref{hab-l612} implies by Theorem \ref{p-basis-PaHk}:

\begin{proposition}\label{hab-s616}
Assume (H1--OP$_{F,a,H,k}$) and $k\in \operatorname*{int}0^+H$. Then: 
\begin{itemize}
\item[(a)] $F$ is the feasible range of each problem $(P_{F,b,H,k^0})$ with $b\in Y$, $k^0\in\operatorname*{int}0^+H$.
\item[(b)] If the objective function of $(P_{F,a,H,k})$ is not bounded below on $F$, then none of the problems $(P_{F,b,H,k^0})$ with $b\in Y$ and $k^0\in\operatorname*{int}0^+H$ has an optimal solution.
\end{itemize}
\end{proposition}

\begin{proposition}\label{hab-s618} 
Assume $F\subseteq \mathbb{R}^{\ell}$ is a nonempty closed convex set, $a\in \mathbb{R}^{\ell}$, $H$ is a proper closed convex subset of $\mathbb{R}^{\ell}$, and $k\in 0^+H\setminus (-0^+H)$. 
\begin{itemize}
\item[(a)] If $M_{F,a,H,k}=\emptyset$ and $F\cap (a-H+\mathbb{R}k)\not=\emptyset$ or if $M_{F,a,H,k}$ is not bounded, then, for each $b\in \mathbb{R}^{\ell}$ and $k^0\in 0^+H\setminus (-0^+H)$, $M_{F,b,H,k^0}=\emptyset$ or 
$\,0^+M_{F,b,H,k^0}\not=\{0\}$.
\item[(b)] If $M_{F,a,H,k}$ is nonempty and bounded, then, for each $b\in \mathbb{R}^{\ell}$ and $k^0\in 0^+H\setminus (-0^+H)$ with $F\cap (b-H+\mathbb{R}k^0)\not=\emptyset$, $M_{F,b,H,k^0}$ is nonempty and compact. 
\end{itemize}
\end{proposition}

\begin{proof}
By Lemma \ref{t251}(f), each function $\varphi _{b-H,k^0}$ with $b\in Y$ and $k^0\in 0^+H\setminus (-0^+H)$ is proper.
\begin{itemize}
\item[(a)] 
Consider first the case that $M_{F,a,H,k}=\emptyset$ and that
there exists some $t\in \mathbb{R}$ with $F\cap (a+tk-H)\not=\emptyset$.
$F\cap (a+tk-H)$ cannot be compact, since otherwise $\varphi _{a-H,k}$ would attain a minimum on $F\cap (a+tk-H)$ und thus on $F\cap \operatorname*{dom}\varphi _{a-H,k}$. Since $F\cap (a+tk-H)$ is closed, it is not bounded.\\
Consider now the case that $M_{F,a,H,k}$ is not bounded.
Let $t$ be the minimum of $\varphi _{a-H,k}$ on $F$.
$\Rightarrow F\cap (a+tk-H)$ is not bounded.\\
For $b\in \mathbb{R}^{\ell}$, $k^0\in 0^+H\setminus (-0^+H)$, if the minimum $\bar{t}$ of $\varphi _{b-H,k^0}$ on $F$ exists, then
$M_{F,b,H,k^0}=F\cap (b+\bar{t} k^0-H)$, and $0^+M_{F,b,H,k^0}\neq \{ 0 \} $ by Lemma \ref{hab-s21_16}.
\item[(b)] follows from (a).
\end{itemize}
\end{proof}

Note that the condition $F\cap (a-H+\mathbb{R}k)\not=\emptyset$ is equivalent to the existence of feasible solutions of problem $(P_{F,a,H,k})$.

\begin{proposition}\label{hab-s617} 
Assume $F\subseteq \mathbb{R}^{\ell}$ is a nonempty closed convex set, $a\in \mathbb{R}^{\ell}$, $H$ is a proper closed convex subset of $\mathbb{R}^{\ell}$ 
and $k\in \operatorname*{int}0^+H$. 
\begin{itemize}
\item[(a)] If $M_{F,a,H,k}=\emptyset$ or if $M_{F,a,H,k}$ is not bounded, then, for each $b\in \mathbb{R}^{\ell}$ and $k^0\in 0^+H\setminus (-0^+H)$, $M_{F,b,H,k^0}=\emptyset$ or 
$\,0^+M_{F,b,H,k^0}\not=\{0\}$.
\item[(b)] If $M_{F,a,H,k}$ is nonempty and bounded, then, for each $b\in \mathbb{R}^{\ell}$ and $k^0\in\operatorname*{int}0^+H$, $M_{F,b,H,k^0}$ is nonempty and compact. 
\end{itemize}
\end{proposition}

\begin{proof}
By Lemma \ref{t251}(d), each function $\varphi _{b-H,k^0}$ with $b\in Y$ and $k^0\in\operatorname*{int}0^+H$ is finite-valued. Apply Proposition \ref{hab-s618}.
\end{proof}

If $(P_{F,a,H,k})$ in Proposition \ref{hab-s617}(b) has a unique minimal solution, then this is not necessarily the case for the problems $(P_{F,b,H,k^0})$ considered.

\begin{example}\label{hab-ex618}
Choose $F=H=\mathbb{R}_+^2$, $a=(0,0)^T$, $b=(1,0)^T$,
$k=(1,1)^T$.
Then $F\cap (a-H)=(0,0)^T$ and
$F\cap (b-H)=F\cap (b-\operatorname*{bd}H)=\{ (y_1,y_2)^T\in \mathbb{R}^2\mid 0\leq y_1\leq 1,\, y_2=0\}$.
Thus, $\varphi _{a-H,k}$ attains its minimum on $F$ only at $(0,0)^T$, whereas
each element of $F\cap (b-H)$ is a minimizer of $\varphi _{b-H,k}$ on $F$.
\end{example}

Further sensitivity results for the studied problems $(P_{F,a,H,k})$ in $Y=\mathbb{R}^{\ell}$ were proved by Pascoletti and Serafini \cite{pase84} for polyhedral cones $H$, later by Helbig for closed convex cones $H$ \cite{helb92b} as well as for closed convex sets \cite{helb92a} and by Eichfelder \cite{eich08} for closed convex pointed cones $H$. 
The way in which the solution set of problems $(P_{F,a,H,k})$ depends on alternations in $a$ and $k$ had also been studied for convex cones $H$ in more general spaces $Y$ by Sterna-Karwat in \cite{ster87a} and \cite{ster87b}.
In \cite{eich08}, further references related to sensitivity are given. 

\def\cfac#1{\ifmmode\setbox7\hbox{$\accent"5E#1$}\else
  \setbox7\hbox{\accent"5E#1}\penalty 10000\relax\fi\raise 1\ht7
  \hbox{\lower1.15ex\hbox to 1\wd7{\hss\accent"13\hss}}\penalty 10000
  \hskip-1\wd7\penalty 10000\box7}
  \def\cfac#1{\ifmmode\setbox7\hbox{$\accent"5E#1$}\else
  \setbox7\hbox{\accent"5E#1}\penalty 10000\relax\fi\raise 1\ht7
  \hbox{\lower1.15ex\hbox to 1\wd7{\hss\accent"13\hss}}\penalty 10000
  \hskip-1\wd7\penalty 10000\box7}
  \def\cfac#1{\ifmmode\setbox7\hbox{$\accent"5E#1$}\else
  \setbox7\hbox{\accent"5E#1}\penalty 10000\relax\fi\raise 1\ht7
  \hbox{\lower1.15ex\hbox to 1\wd7{\hss\accent"13\hss}}\penalty 10000
  \hskip-1\wd7\penalty 10000\box7}
  \def\cfac#1{\ifmmode\setbox7\hbox{$\accent"5E#1$}\else
  \setbox7\hbox{\accent"5E#1}\penalty 10000\relax\fi\raise 1\ht7
  \hbox{\lower1.15ex\hbox to 1\wd7{\hss\accent"13\hss}}\penalty 10000
  \hskip-1\wd7\penalty 10000\box7}
  \def\cfac#1{\ifmmode\setbox7\hbox{$\accent"5E#1$}\else
  \setbox7\hbox{\accent"5E#1}\penalty 10000\relax\fi\raise 1\ht7
  \hbox{\lower1.15ex\hbox to 1\wd7{\hss\accent"13\hss}}\penalty 10000
  \hskip-1\wd7\penalty 10000\box7}
  \def\cfac#1{\ifmmode\setbox7\hbox{$\accent"5E#1$}\else
  \setbox7\hbox{\accent"5E#1}\penalty 10000\relax\fi\raise 1\ht7
  \hbox{\lower1.15ex\hbox to 1\wd7{\hss\accent"13\hss}}\penalty 10000
  \hskip-1\wd7\penalty 10000\box7}
  \def\cfac#1{\ifmmode\setbox7\hbox{$\accent"5E#1$}\else
  \setbox7\hbox{\accent"5E#1}\penalty 10000\relax\fi\raise 1\ht7
  \hbox{\lower1.15ex\hbox to 1\wd7{\hss\accent"13\hss}}\penalty 10000
  \hskip-1\wd7\penalty 10000\box7}
  \def\cfac#1{\ifmmode\setbox7\hbox{$\accent"5E#1$}\else
  \setbox7\hbox{\accent"5E#1}\penalty 10000\relax\fi\raise 1\ht7
  \hbox{\lower1.15ex\hbox to 1\wd7{\hss\accent"13\hss}}\penalty 10000
  \hskip-1\wd7\penalty 10000\box7}
  \def\cfac#1{\ifmmode\setbox7\hbox{$\accent"5E#1$}\else
  \setbox7\hbox{\accent"5E#1}\penalty 10000\relax\fi\raise 1\ht7
  \hbox{\lower1.15ex\hbox to 1\wd7{\hss\accent"13\hss}}\penalty 10000
  \hskip-1\wd7\penalty 10000\box7} \def\Dbar{\leavevmode\lower.6ex\hbox to
  0pt{\hskip-.23ex \accent"16\hss}D}
  \def\cfac#1{\ifmmode\setbox7\hbox{$\accent"5E#1$}\else
  \setbox7\hbox{\accent"5E#1}\penalty 10000\relax\fi\raise 1\ht7
  \hbox{\lower1.15ex\hbox to 1\wd7{\hss\accent"13\hss}}\penalty 10000
  \hskip-1\wd7\penalty 10000\box7} \def\cprime{$'$}
  \def\Dbar{\leavevmode\lower.6ex\hbox to 0pt{\hskip-.23ex \accent"16\hss}D}
  \def\cfac#1{\ifmmode\setbox7\hbox{$\accent"5E#1$}\else
  \setbox7\hbox{\accent"5E#1}\penalty 10000\relax\fi\raise 1\ht7
  \hbox{\lower1.15ex\hbox to 1\wd7{\hss\accent"13\hss}}\penalty 10000
  \hskip-1\wd7\penalty 10000\box7} \def\cprime{$'$}
  \def\Dbar{\leavevmode\lower.6ex\hbox to 0pt{\hskip-.23ex \accent"16\hss}D}
  \def\cfac#1{\ifmmode\setbox7\hbox{$\accent"5E#1$}\else
  \setbox7\hbox{\accent"5E#1}\penalty 10000\relax\fi\raise 1\ht7
  \hbox{\lower1.15ex\hbox to 1\wd7{\hss\accent"13\hss}}\penalty 10000
  \hskip-1\wd7\penalty 10000\box7}
  \def\udot#1{\ifmmode\oalign{$#1$\crcr\hidewidth.\hidewidth
  }\else\oalign{#1\crcr\hidewidth.\hidewidth}\fi}
  \def\cfac#1{\ifmmode\setbox7\hbox{$\accent"5E#1$}\else
  \setbox7\hbox{\accent"5E#1}\penalty 10000\relax\fi\raise 1\ht7
  \hbox{\lower1.15ex\hbox to 1\wd7{\hss\accent"13\hss}}\penalty 10000
  \hskip-1\wd7\penalty 10000\box7} \def\Dbar{\leavevmode\lower.6ex\hbox to
  0pt{\hskip-.23ex \accent"16\hss}D}
  \def\cfac#1{\ifmmode\setbox7\hbox{$\accent"5E#1$}\else
  \setbox7\hbox{\accent"5E#1}\penalty 10000\relax\fi\raise 1\ht7
  \hbox{\lower1.15ex\hbox to 1\wd7{\hss\accent"13\hss}}\penalty 10000
  \hskip-1\wd7\penalty 10000\box7} \def\Dbar{\leavevmode\lower.6ex\hbox to
  0pt{\hskip-.23ex \accent"16\hss}D}
  \def\cfac#1{\ifmmode\setbox7\hbox{$\accent"5E#1$}\else
  \setbox7\hbox{\accent"5E#1}\penalty 10000\relax\fi\raise 1\ht7
  \hbox{\lower1.15ex\hbox to 1\wd7{\hss\accent"13\hss}}\penalty 10000
  \hskip-1\wd7\penalty 10000\box7} \def\Dbar{\leavevmode\lower.6ex\hbox to
  0pt{\hskip-.23ex \accent"16\hss}D}
  \def\cfac#1{\ifmmode\setbox7\hbox{$\accent"5E#1$}\else
  \setbox7\hbox{\accent"5E#1}\penalty 10000\relax\fi\raise 1\ht7
  \hbox{\lower1.15ex\hbox to 1\wd7{\hss\accent"13\hss}}\penalty 10000
  \hskip-1\wd7\penalty 10000\box7}

\end{document}